\documentclass[preprint, 12pt, english]{elsarticle}
\usepackage{amsmath}
\usepackage{latexsym, amssymb}
\usepackage{amsthm}
\usepackage{color}
\usepackage{txfonts}

\newtheorem{thm}{Theorem}[section] 

\newtheorem{cor}[thm]{Corollary}

\newtheorem{lem}[thm]{Lemma}
\newtheorem{prop}[thm]{Proposition}

\newtheorem{rem}[thm]{Remark}

\newcommand\operA[2]{{\if!#2!\operatorname{#1}\else{\operatorname{#1}_{#2}^{\phantom{I}}}\fi}} 
\newcommand\set[1]{\{#1\}}

\newcommand\Cref[1]{{Corollary~\ref{#1}}}%

\def\tr{{\operatorname{Tr}}}

\def\CX{\mathcal X}

\newcommand{\Trace}[1][]{\if!#1!\operatorname{Tr}\else{\operatorname{Tr}_{#1}^{\phantom{I}}}\fi} 

\long\def\forget#1\forgotten{{}} %

\def\({\left(}
\def\){\right)}

\newif\iffurther
\furtherfalse

\newif\ifXY 
\XYtrue     
\ifXY

\input xy
\input xyidioms.tex
\usepackage{xy}
\xyoption{all} %
\fi 

\usepackage{babel}

\journal{??}

\begin{document}

\begin{frontmatter}

\title{Kummer Elements in Cyclic Algebras of Degree 5}

\author{Adam Chapman\corref{ch}}
\ead{adam1chapman@yahoo.com}
\cortext[ch]{The author is supported by Wallonie-Bruxelles International.}
\address{ICTEAM Institute, Universit\'{e} Catholique de Louvain, B-1348 Louvain-la-Neuve, Belgium.}

\begin{abstract}
We construct a graph of Kummer elements in a given cyclic algebra of prime degree and study its properties.
In case of degree 5, we provide sufficient conditions for two elements to have a chain of Kummer elements connecting them, such that the multiplicative commutator of any two consecutive elements in the chain is a root of unity.
\end{abstract}


\end{frontmatter}

\section{Introduction}

One of the major objects of interest in modern algebra is the Brauer group of a field.
In order to understand the structure of this group, we want to develop methods for determining wether two elements in this group are equal.
The group is known to be a torsion group.
An element of order $d$ is the Brauer equivalence class of some tensor product of cyclic algebras of degree $d$, according to \cite{MS}, assuming the characteristic of the field $F$ is prime to $d$ and that the field contains a primitive $d$th root of unity $\rho$.
Under these assumptions, a cyclic algebra of degree $p$ has the form
$$(\alpha,\beta)_{p,F}=F[x,y : x^p=\alpha, y^p=\beta, y x=\rho x y]$$
for some $\alpha,\beta \in F^\times$.

The question is therefore whether two given tensor products of cyclic algebras of degree $d$ are Brauer equivalent.
By multiplying the shorter tensor product by a matrix algebra, we can make the tensor products be of the same length, in which case being Brauer equivalent is the same as being isomorphic.
One way to approach this problem is to come up with a set of basic steps with which one can obtain from one given tensor product of cyclic algebras all the other tensor products isomorphic to it.
A theorem providing such as set is usually referred to as a ``chain lemma" or ``chain equivalence".
In case of quaternion algebras ($d=2$), there are known chain lemmas for one quaternion algebra, tensor products of two quaternion algebras (\cite{ChapVish2} and \cite{Siv}) and tensor products of any number of quaternion algebras over fields of cohomological 2-dimension 2 (\cite{ChapmanCE}).
A chain lemma for one cyclic algebra of degree 3 was provided in \cite{Rost}.

In the papers \cite{ChapVish2} and \cite{HKT} (where an alternative proof to the result from \cite{Rost} was provided) the key-idea was to show that every two Kummer elements, i.e. elements whose $d$th power are their minimal central powers, are connected by a chain of Kummer elements such that the multiplicative commutator of every two consecutive elements in the chain is a power of $\rho$.
At least when dealing with one cyclic algebra of prime degree (or a tensor product of two quaternion algebras) such a theorem implies that every two isomorphic symbols $(\alpha,\beta)$ and $(\alpha',\beta')$ are connected by a chain of isomorphic symbols such that every two consecutive symbols share a common slot.

In this paper we follow this approach by constructing a graph that encodes the relations between different Kummer elements and studying its properties.
This way we manage to provide some sufficient conditions for a pair of Kummer space to be connected by such a chain as described above in case of $d=5$.

\section{Notation}

We wish to introduce here a couple of expressions that will be used later on in this paper:
The (additive) commutator is $[x,z]_d=z x-\rho^d x z$, and the multi-commutator is defined recursively as to be $$[x_1,\dots,x_{k+1}]_{d_1,d_2,\dots,d_k}=[x_1,\dots,x_{k-1},[x_k,x_{k+1}]_{d_1}]_{d_2,\dots,d_k}.$$
The sum of all the words in which the letters $x_1,\dots,x_k$ appear $d_1,\dots,d_k$ times (respectively) is denoted by $x_1^{d_1} * x_2^{d_2} * \dots * x_n^{d_n}$. For example $x^2 * y=x^2 y+x y x+y x^2$. (This notation was introduced in \cite{Revoy}.)

\section{The graph}

Fix a prime $p$, a field $F$ of characteristic not $p$ containing a $p$th root of unity $\rho$ and a cyclic algebra $A$ of degree $p$ over $F$.
Let $\mathcal{X}$ be the set of all Kummer elements in $A$, i.e. all $x \in A \setminus F$ satisfying $x^p \in F^\times$.
We construct a complete labeled digraph $(\mathcal{X},E)$ where the label of each edge $\xymatrix@C=20px@R=20px{x \ar@{->}[r]  & z}$ is the set $\set{0 \leq i \leq p-1 : z_i \neq 0}$ such that $z_0+z_1+\dots+z_{p-1}$ is the unique decomposition of $z$ with $z_i x=\rho^i x z_i$ for each $0 \leq i \leq p-1$. (See \cite[Corollary 3.4]{ChapVish1} for the existence and uniqueness of this decomposition.)
We denote the label of this edge by $l(x,z)$.
The weight of this edge is $\sharp l(x,z)$ and we denote it by $w(x,z)$.
Occasionally we simply write the elements with an edge connecting them, where the label (or weight) of the edge appears above it, unless we want to indicate the labels in both directions, in which case the label above refers to the edge going from left to right and the label below refers to the edge going from right to left. For example, $\xymatrix@C=20px@R=20px{x \ar@{<->}[rr]^{\set{a_1,\dots,a_k}}_{\set{b_1,\dots,b_m}} &  & z}$ means $l(x,z)=\set{a_1,\dots,a_k}$ and $l(z,x)={\set{b_1,\dots,b_m}}$, and
 $\xymatrix@C=20px@R=20px{x \ar@{<->}[r]^{k}_{m}  & z}$ means $w(x,z)=k$ and $w(z,x)=m$.

It is important to note that $w(x,z)=1$ if and only if $w(z,x)=1$. In that case we often simply write $\xymatrix@C=20px@R=20px{x \ar@{<->}[r] & z}$. That does not hold for greater weights.
For instance, if $p=3$ and $$A=(\alpha,\beta)_{3,F}=F[x,y : x^3=\alpha, y^3=\beta, y x=\rho x y]$$ then for $z=y+x^2 y^2$ we have
$$x=(-\rho \beta \alpha-\rho^2 \alpha^{-1}) (z-x^2 y z-(\rho^2 \alpha^2 \beta)^{-1} (x^2 y)^2 z),$$ which means that
\begin{eqnarray*}
x_0&=&(-\rho \beta \alpha-\rho^2 \alpha^{-1}) z,\\ x_1&=&(-\rho \beta \alpha-\rho^2 \alpha^{-1}) (-x^2 y), \\x_2&=&(-\rho \beta \alpha-\rho^2 \alpha^{-1}) (-(\rho^2 \alpha^2 \beta)^{-1} (x^2 y)^2 z),
\end{eqnarray*}
and therefore $w(x,z)=2 \neq 3=w(z,x)$.

\section{Basic properties}

We wish to list here some basic properties of this graph:

\begin{prop}\label{weight22}
If $\xymatrix@C=20px@R=20px{x \ar@{<->}[r]^{2}_{2}  & z}$ then
\begin{enumerate}
\item $l(z,x)=-l(x,z)$, i.e. $x=x_i+x_j$ and $z=z_{-i}+z_{-j}$ for some $i \neq j$.
\item $x_i x_j=\rho^{j-i} x_j x_i$ and $z_{-i} z_{-j}=\rho^{i-j} z_{-j} z_{-i}$.
\end{enumerate}
\end{prop}

\begin{proof}
We have $x=x_i+x_j$ and $z=z_m+z_n$ such that $z_m x=\rho^m x z_m$, $z_n x=\rho^n x z_n$, $x_i z=\rho^i z x_i$ and $x_j z=\rho^j z x_j$.

Let us consider the equality $[x,x,z]_{m,n}=0$. This holds because $z=z_m+z_n$.
On the other hand, if we substitute $x=x_i+x_j$ in this expression we get the following set of equations (by conjugation by $z$):
\begin{enumerate}
\item $(z x_i-\rho^m x_i z) x_i-\rho^n x_i (z x_i-\rho^m x_i z)=0$
\item $(z x_j-\rho^m x_j z) x_j-\rho^n x_j (z x_j-\rho^m x_j z)=0$
\item $(z x_j-\rho^m x_j z) x_i-\rho^n x_i (z x_j-\rho^m x_j z)+(z x_i-\rho^m x_i z) x_j-\rho^n x_j (z x_i-\rho^m x_i z)=0$
\end{enumerate}

From the first equation we obtain $(\rho^{-i}-\rho^m)(\rho^{-i}-\rho^n) x_i^2 z=0$ and from the second equation we obtain $(\rho^{-j}-\rho^m)(\rho^{-j}-\rho^n) x_i^2 z=0$. Therefore, without loss of generality $m=-i$ and $n=-j$.

From the third equation we obtain $(\rho^{-j}-\rho^{-i}) \rho^{-i} x_j x_i z-(\rho^{-j}-\rho^{-i}) \rho^{-j} x_i x_j z=0$. Consequently, $x_i x_j=\rho^{j-i} x_j x_i$.
Due to symmetry, we also have $z_{-i} z_{-j}=\rho^{i-j} z_{-j} z_{-i}$.
\end{proof}

\begin{lem} \label{lem3}
Assume $\xymatrix@C=20px@R=20px{x \ar@{->}[rr]^{\set{i,j,k}} & & z}$. Then $p=3$ if and only if $2 i \equiv j+k \pmod{p}$ and $2 j \equiv i+k \pmod{p}$.
\end{lem}

\begin{proof}
If $2 i \equiv j+k \pmod{p}$ and $2 j \equiv i+k \pmod{p}$ then $3 (i-j) \equiv 0 \pmod{p}$, but $i-j \not \equiv 0 \pmod{p}$, and therefore $3 \equiv 0 \pmod{3}$, which means that $p=3$.
The other direction is trivial.
\end{proof}

\begin{thm}\label{nothree}
If $\xymatrix@C=20px@R=20px{x \ar@{<->}[r]^{2}_{3}  & z}$ then $x$ and $z$ are connected by edges of weight 1.
\end{thm}

\begin{proof}
For some $m,n,i,j,k \in \mathbb{Z}/p\mathbb{Z}$, $x=x_i+x_j+x_k$ and $z=z_m+z_n$.
For $p=2,3$ it is known to be true, so assume $p>3$.
Then at least two of the following hold: $i+k \not \equiv 2 j \pmod{p}$, $j+k \not \equiv 2 i \pmod{p}$, $i+j \not \equiv 2 k \pmod{p}$, because otherwise $p=3$ according to Lemma \ref{lem3}.
Without loss of generality we may assume that $i+k \not \equiv 2 j \pmod{p}$ and $j+k \not \equiv 2 i \pmod{p}$.

Consider the equality $[x,x,z]_{m,n}=0$.
If we take only the part which $\rho^{2 i}$-commutes with $z$ we get $(\rho^{-i}-\rho^m)(\rho^{-i}-\rho^n) x_i^2 z=0$, and if we take only the part which $\rho^{2 j}$-commutes with $z$ we get $(\rho^{-j}-\rho^m)(\rho^{-j}-\rho^n) x_j^2 z=0$.
Consequently, $m=-i$ and $n=-j$ without loss of generality.
If we take only the part which $\rho^{i+k}$-commutes with $z$ we obtain $(\rho^{-k}-\rho^{-i}) \rho^{-i} x_k x_i z-(\rho^{-k}-\rho^{-i}) \rho^{-k} x_i x_k z=0$. Consequently, $x_i x_k=\rho^{k-i} x_k x_i$. Similarly $x_j x_k=\rho^{k-j} x_k x_j$.

Now consider the equality $[z,z,x]_{i,j}=(\rho^k-\rho^i) (\rho^k-\rho i) z^2 x_k$.
If we substitute $z=z_{-i}+z_{-j}$ on the left-hand side of this equation then we get $(\rho^j-\rho^i) \rho^i z_{-j} z_{-i}-(\rho^j-\rho^i) \rho^j z_{-j} z_{-i}=(\rho^k-\rho^i) (\rho^k-\rho i) x_k$.

Consequently, $x_k$ $\rho^{-i-j}$-commutes with $x$. 
If $k=-i-j=0$ then either $x_k \in F$ or $x$ and $z$ commute.
If $x$ and $z$ commute we are done.
If $x_k \in F$ then $\tr(x)=p x_k=0$ which means that $x_k=0$, contradiction.
Assume then that either $k \neq 0$ or $-i-j \neq 0$.
Then $x_k \in \mathcal{X}$ and $\xymatrix@C=20px@R=20px{x \ar@{<->}[r] & x_k \ar@{<->}[r] & z}$.
\end{proof}

\begin{prop} \label{split2cent}
If $\xymatrix@C=20px@R=20px{x \ar@{->}[rr]^{\set{i,j}} &  & z}$, then $z_i,z_j \in \CX$.
\end{prop}

\begin{proof}
From $F \ni z^p=(z_i+z_j)^p=\sum_{k=0}^p z_i^k * z_j^{p-k}$, by taking the part of the right-hand side of this equality which commutes with $x$ we obtain $z^p=z_i^p+z_j^p$.
If $i,j \neq 0$ then automatically $z_i$ and $z_j$ are Kummer.
If $i=0$ then $j \neq 0$ and so $z_j \in \CX$, therefore $z_j^p \in F$, and consequently $z_i^p=z^p-z_j^p \in F$,
which means that either $z_i \in \CX$ or $z_i \in F$.
The second option is not possible, because then $\tr(z)=p z_i=0$, contradiction.
\end{proof}

\begin{cor}\label{pcent}
If $\xymatrix@C=20px@R=20px{x \ar@{->}[rr]^{\set{i,j}} &  & z}$ then $F z_i+F z_j$ is a Kummer space. Moreover, its exponentiation form $f(u,v)=(u z_i+v z_j)^p$ is diagonal, i.e. $f(u,z)=u^p z_i^p+v^p z_j^p$.
\end{cor}

\begin{proof}
We have $z^p=(z_i+z_j)^p$. By conjugation by $x$ we obtain the relations $z_i^k * z_j^{p-k}=0$ for all $1 \leq k \leq p-1$.
Together with the result from Proposition \ref{split2cent}, the space $F z_i+F z_j$ is therefore Kummer and it is easy to see why the exponentiation form is diagonal.
\end{proof}

\begin{cor}\label{nozero}
If $\xymatrix@C=20px@R=20px{x \ar@{->}[rr]^{\set{i,j}} &  & z}$ then $0 \not \in l(z_i,z_j)$. In particular $w(z_i,z_j) \leq p-1$.
\end{cor}

\begin{proof}
The exponentiation form of $F z_i+F z_j$ is diagonal, therefore $z_i^{p-1} * z_j=0$. However, $z_i^{p-1} * z_j=p z_i^{p-1} z_{j,0}$, which means that $z_{j,0}=0$.
\end{proof}

\begin{rem} \label{notzero}
If $\xymatrix@C=20px@R=20px{x \ar@{->}[rr]^{\set{i,j}} &  & z}$ and $\xymatrix@C=20px@R=20px{z_i \ar@{<->}[r] & z_j}$ then $\xymatrix@C=20px@R=20px{x \ar@{<->}[r] & z_i z_j^{-1} \ar@{<->}[r] & z}$.
\end{rem}

\begin{cor} \label{notzerocor}
If $\xymatrix@C=20px@R=20px{x \ar@{->}[rr]^{\set{0,j}} &  & z}$ then $\xymatrix@C=20px@R=20px{z_0 \ar@{<->}[r] & z_j}$. Moreover, $x$ and $z$ are connected by two edges of weight 1.
\end{cor}

\begin{proof}
Since $A$ is cyclic of degree $p$, and $z_0 \in \CX$ (according to \ref{split2cent}), $z_0=x^k$ for some $k$. Therefore $w(z_0,z_j)=1$. As a Result of Remark \ref{notzero}, $x$ and $z$ are connected by two edges of weight 1.
\end{proof}

\begin{prop}\label{prop2in2}
If $\xymatrix@C=20px@R=20px{x \ar@{->}[rr]^{\set{i,j}} &  & z}$ and $\xymatrix@C=20px@R=20px{z_i \ar@{->}[rr]^{\set{m,n}} &  & z_j}$ then
$m \not \equiv -n \pmod{p}$.
\end{prop}

\begin{proof}
According to Corollary \ref{notzerocor}, $i \neq 0$ because $w(z_i,z_j)=2$. Therefore $z_{j,m} \in F z_i^{j i^{-1}} x^{m (-i)^{-1}}$ and $z_{j,n} \in F z_i^{j i^{-1}} x^{n (-i)^{-1}}$.
We have $m,n \neq 0$ due to Corollary \ref{nozero}.
Assume $m \equiv -n \pmod{p}$.

In Corollary \ref{pcent} we saw that $F z_i+F z_j$ is a Kummer space with a diagonal exponentiation form. Since $l(z_i,z_j)=\set{m,-m}$.
From \cite{ChapVish1} we obtain $z_{j,m} z_{j,n}=\rho^m z_{j,n} z_{j,m}$.
Consequently $\rho^{i^{-2} j (m-n)}=\rho^m$, which means that $2 j \equiv i^2 \pmod{p}$.
But now, $l(x^2,z)=\set{2 i,2 j}$ and therefore for similar arguments $2 (2 j) \equiv (2 i)^2 \pmod{p}$, and that creates a contradiction.
\end{proof}

\begin{thm}\label{twotwo}
If $\xymatrix@C=20px@R=20px{x \ar@{<->}[r]^{2}_{2}  & z}$ then $x$ and $z$ are connected by two edges of weight 1.
\end{thm}

\begin{proof}
We have $x=x_i+x_j$ and $z=z_{-i}+z_{-j}$ according to Proposition \ref{weight22}. Setting $y=x z-\rho^i z x$ we have $y=(\rho^j-\rho^i) z x_j$. Therefore $y^p=(\rho^j-\rho^i)^p z^p x_j^p$. According to Proposition \ref{split2cent} we know that $x_j \in \CX$, and therefore $y \in \CX$. Since $w(x,y)=w(z,y)=1$ we get the chain $\xymatrix@C=20px@R=20px{x \ar@{<->}[r]  & y \ar@{<->}[r]  & z}$.
\end{proof}

\section{In cyclic algebras of degree 5}

Fix $p=5$.
The main goal of this section is to prove the following:
\begin{thm}\label{twonozero}
If $w(x,z)=2$ and $0 \not \in l(z,x)$ then $x$ and $z$ are connected by a sequence of edges of weight 1.
\end{thm}

The proof of the theorem occupies the rest of this section.

\begin{prop}
If $\xymatrix@C=20px@R=20px{x \ar@{->}[rr]^{\set{i,j}} &  & z}$ and $\xymatrix@C=20px@R=20px{z_i \ar@{->}[r]^{2}  & z_j}$ then $x$ and $z$ are connected by three edges of weight 1.
\end{prop}

\begin{proof}
Let $l(x,z)=\set{i,j}$ and $l(z_i,z_j)=\set{m,n}$.
Without loss of generality we can assume that $m=1$ and $n=3$ or $n=4$.
According to Proposition \ref{prop2in2}, the case of $n=4$ is not possible, and so we assume that $n=3$.
In \cite{ChapVish1} it is proven that in this case either $z_{j,1} z_{j,3}=\rho z_{j,3} z_{j,1}$ or $z_{j,1} z_{j,3}=\rho^2 z_{j,3} z_{j,1}$.

In the first case the chain is
$$\xymatrix@C=20px@R=20px{z \ar@{<->}[r]  & z_{j,1}^{-1} (z_i+z_{j,3}) \ar@{<->}[r]  & z_{j,1} \ar@{<->}[r]  & x}.$$ In the second case the chain is
$$\xymatrix@C=20px@R=20px{z \ar@{<->}[r]  & z_{j,3}^{-1} (z_i+z_{j,1}) \ar@{<->}[r]  & z_{j,3} \ar@{<->}[r]  & x}.$$
\end{proof}

\begin{thm}
If $\xymatrix@C=20px@R=20px{x \ar@{->}[rr]^{\set{i,j}} &  & z}$ then $w(z_i,z_j) \neq 3$.
\end{thm}

\begin{proof}
If $l(z_i,z_j)=\set{m,n,k}$ then without loss of generality $m \equiv -n \pmod{5}$ and $m,n \not \equiv -k \pmod{5}$.

The space $F z_i+F z_j$ is a Kummer with a diagonal exponentiation form according to Corollary \ref{pcent}.
From the relation $(z_i)^3 * (z_j)^2=0$ we get $z_{j,m} z_{j,n}=\rho^m z_{j,n} z_{j,m}$. But again, as in the proof of Proposition \ref{prop2in2}, it means that $2 j \equiv i^2 \pmod{p}$. As before, we shall have a contradiction, because we get $2 (2 j) \equiv (2 i)^2 \pmod{p}$ as well.
\end{proof}

There are however examples where $\xymatrix@C=20px@R=20px{x \ar@{->}[rr]^{\set{i,j}} &  & z}$ and $\xymatrix@C=20px@R=20px{z_i \ar@{->}[r]^{2}  & z_j}$ or $\xymatrix@C=20px@R=20px{z_i \ar@{->}[r]^{4}  & z_j}$:

\begin{rem}
If $A=F[x,y : x^5=\alpha, y^5=\beta, y x=\rho x y]$ and $z=y+(a_1 x+a_2 x^2+a_3 x^3+a_4 x^4) y^{-1}$ then $z$ is Kummer if and only if $a_2 a_3=(\rho^4-\rho) a_1 a_4$. Consequently, if we take $a_3=a_4=0$ then $w(x,z)=2$ and $w(z_1,z_{-1})=2$, and if we take $a_1=a_4=a_3=1$ and $a_2=(\rho^4-\rho)$ then $w(x,z)=2$ while $w(z_1,z_{-1})=4$.
\end{rem}

\begin{proof}
Consider the relations $z_1 * z_4^4=z_1^2 * z_4^3=z_1^3 * z_4^2=z_1^4 * z_4=0$ ($z_1=y$ and $z_4=a_0+a_1 x+a_2 x^2+a_3 x^3+a_4 x^4) y^{-1}$). On one hand, these relations are satisfied if and only if $z \in \CX$ (Corollary \ref{pcent}). On the other hand, it can be checked that these relations are satisfied if and only if $a_0=0$ and $a_2 a_3=(\rho^4-\rho) a_1 a_4$:

Due to the relation $y_1^4 * y_4=0$ we have $a_0=0$. Write $w_i=a_i x^i y_1^{-1}$.

Now, the relation $y_1^3 * y_4^2=0$ provides the following by conjugation by $y_1$:
\begin{enumerate}
\item $y_1^3 * w_3^2+y_1^3 * w_2 * w_4=0$
\item $y_1^3 * w_1^2+y_1^3 * w_3 * w_4=0$
\item $y_1^3 * w_4^2+y_1^3 * w_1 * w_2=0$
\item $y_1^3 * w_2^2+y_1^3 * w_1 * w_3=0$
\item $y_1^3 * w_1 * w_4+y_1^3 * w_2 * w_3=0$
\end{enumerate}

The first four relations are trivial.
From the fifth we obtain $5 (\rho+1+\rho^{-1}) a_1 a_4 \alpha y_1+5 (\rho^3+\rho^2+1) a_2 a_3 \alpha y_1=0$.
Consequently, $a_2 a_3=(\rho^4-\rho) a_1 a_4$.

The relation  $y_1^2 * y_4^3$ provides the following by conjugation by $y_1$:
\begin{enumerate}
\item $y_1^2 * w_1 * w_2^2+y_1^2 * w_1^2 * w_3+y_1^2 * w_2 * w_4^2+y_1^2 * w_3^2 * w_4=0$
\item $y_1^2 * w_1 * w_2 * w_3+y_1^2 * w_1^2 * w_4+y_1^2 * w_2^3+y_1^2 * w_3 * w_4^2=0$
\item $y_1^2 * w_1 * w_2 * w_4+y_1^2 * w_2^2 * w_3+y_1^2 * w_4^3+y_1^2 * w_1 w_3^2=0$
\item $y_1^2 * w_1^3+y_1^2 * w_1 * w_3 * w_4+ y_1^2 * w_2 * w_3^2+y_1^2 * w_2^2 * w_4=0$
\item $y_1^2 * w_3^3+y_1^2 * w_2 * w_3 * w_4+y_1^2 * w_1 * w_4^2+y_1^2 * w_1^2 * w_2=0$
\end{enumerate}

The first relation is trivial.
The second relation implies that $5 (\rho^3+\rho^2+1) a_1 a_2 a_3 \alpha x y^{-1}+5 (\rho+1+\rho^{-1}) a_1^2 a_4 \alpha x y_1^{-1}=0$. This is automatically satisfied given $a_2 a_3=(\rho^4-\rho) a_1 a_4$. The same happens with the succeeding relations.

The relation  $y_1 * y_4^4$ provides the following by conjugation by $y_1$:
\begin{enumerate}
\item $y_1 * w_1 * w_2 * w_3 * w_4+y_1 * w_1^2 * w_4^2+y_1 * w_2^2 * w_3^2+y_1 * w_2^3 * w_4+y_1 * w_3 * w_4^3+y_1 * w_1 * w_3^3+y_1 * w_1^3 * w_2=0$
\item $y_1 * w_1^3 * w_3+y_1 * w_1^2 * w_2^2+y_1 * w_1 * w_2 * w_4^2+y_1 * w_1 * w_3^2 * w_4+y_1 * w_2^2 * w_3 * w_4+y_1 * w_2 * w_3^3+y_1 * w_4^4=0$
\item $y_1 * w_2^3 * w_1+y_1 * w_2^2 * w_4^2+y_1 * w_2 * w_4 * w_3^2+y_1 * w_2 * w_1^2 * w_3+y_1 * w_4^2 * w_1 * w_3+y_1 * w_4 * w_1^3+y_1 * w_3^4=0$
\item $y_1 * w_3^3 * w_4+y_1 * w_3^2 * w_1^2+y_1 * w_3 * w_1 * w_2^2+y_1 * w_3 * w_4^2 * w_2+y_1 * w_1^2 * w_4 * w_2+y_1 * w_1 * w_4^3+y_1 * w_2^4=0$
\item $y_1 * w_4^3 * w_2+y_1 * w_4^2 * w_3^2+y_1 * w_4 * w_3 * w_1^2+y_1 * w_4 * w_2^2 * w_1+y_1 * w_3^2 * w_2 * w_1+y_1 * w_3 * w_2^3+y_1 * w_1^4=0$
\end{enumerate}
All these relations are trivial.
\end{proof}

\begin{proof}[Proof of Theorem \ref{twonozero}]
The case of $0 \in l(x,z)$ has already been dealt with (Corollary \ref{notzerocor}). The same goes for $w(z,x)=2$ (Theorem \ref{twotwo}).
The case of $w(x,z)=2$ and $w(z,x)=3$ was covered in Theorem \ref{nothree}. Assume $w(z,x) \neq 3$.
Let us assume that $0 \not \in l(x,z)$ and $w(z,x)=4$.
Consequently, $l(x,z) \subseteq l(z,x)$.

There are two distinct cases: $l(x,z)=\set{1,4}$ and $l(x,z)=\set{1,3}$.
Assume $l(x,z)=\set{1,4}$. By taking the part of equality $[x,x,z]_{4,1}=0$ which $\rho^3$-commutes with $z$ we obtain $(z x_2-\rho^4 x_2 z) x_1-\rho x_1 (z x_2-\rho^4 x_2 z)=0$. Consequently $(\rho^3-\rho^4) \rho^4 x_2 x_1 z-(\rho^3-\rho^4) \rho x_1 x_2 z=0$, which means that $x_1 x_2=\rho^3 x_2 x_1$.
Therefore $x_2=a x_1^2 z^3$ for some $a \in F$.

Now, by taking the part of the equality $[x,x,z]_{4,1}=0$ which $\rho^4$-commutes with $z$ we obtain $(z x_3-\rho^4 x_3 z) x_1-\rho x_1 (z x_3-\rho^4 x_3 z)+(z x_2-\rho^4 x_2 z) x_2-\rho x_2 (z x_2-\rho^4 x_2 z)=0$. Hence $x_3=(-\rho^3-1) a z^5 x_1^3 z+b x_1^3 z^3$.

By taking the part of the equality $[x,x,z]_{1,4}=0$ which $\rho^2$-commutes with $z$ we obtain $(z x_3-\rho x_3 z) x_4-\rho^4 x_4 (z x_3-\rho x_3 z)=0$. Consequently $(\rho^2-\rho) \rho x_3 x_4 z-(\rho^2-\rho) \rho^4 x_4 x_3 z=0$, which means that $x_3 x_4=\rho^3 x_4 x_3$. Therefore $x_4=c x_3^3 z$ for some $c \in F$.

By taking the part of the equality $[x,x,z]_{1,4}=0$ which $\rho$-commutes with $z$ we obtain $(z x_2-\rho x_2 z) x_4-\rho^4 x_4 (z x_2-\rho x_2 z)+(z x_3-\rho x_3 z) x_3-\rho^4 x_3 (z x_3-\rho x_3 z)=0$. Now, by taking the projection on the line $F x_1 z^4$, we get that $b=0$.

Therefore $x \in F[x_1 z^3] x_1$, which means that
$\xymatrix@C=20px@R=20px{x \ar@{<->}[r]  & x_1 z^3 \ar@{<->}[r]  & z}$.

Assume $l(x,z)=\set{1,3}$. Then we have the following equality $(\rho^4-\rho)(\rho^4-\rho^2)(\rho^4-\rho^3) z^3 x_4=[z,z,z,x]_{1,2,3}$.
Substituting $z=z_1+z_3$ in this equality we get that $l(x,x_4)=\set{0,2}$. Consequently $x$ and $x_4$ are connected by edges of weight 1, and because $w(x_4,z)=1$, so are $x$ and $z$.
\end{proof}

\section*{Acknowledgements}
I owe thanks to Jean-Pierre Tignol and Uzi Vishne for their help and support.

\section*{Bibliography}
\bibliographystyle{amsalpha}
\bibliography{bibfile}

\def\cprime{$'$}
\providecommand{\bysame}{\leavevmode\hbox to3em{\hrulefill}\thinspace}
\providecommand{\MR}{\relax\ifhmode\unskip\space\fi MR }
\providecommand{\MRhref}[2]{%
  \href{http://www.ams.org/mathscinet-getitem?mr=#1}{#2}
}
\providecommand{\href}[2]{#2}
\begin{thebibliography}{HKT09}

\bibitem[Chaar]{ChapmanCE}
Adam Chapman, \emph{Chain equivalences for symplectic bases, quadratic forms
  and tensor products of quaternion algebras}, J. Algebra Appl. (to appear).

\bibitem[CV12]{ChapVish1}
Adam Chapman and Uzi Vishne, \emph{Clifford algebras of binary homogeneous
  forms}, J. Algebra \textbf{366} (2012), 94--111. \MR{2942645}

\bibitem[CV13]{ChapVish2}
\bysame, \emph{Square-central elements and standard generators for biquaternion
  algebras}, Israel J. Math. \textbf{197} (2013), no.~1, 409--423. \MR{3096621}

\bibitem[HKT09]{HKT}
Darrell Haile, Jung-Miao Kuo, and Jean-Pierre Tignol, \emph{On chains in
  division algebras of degree 3}, C. R. Math. Acad. Sci. Paris \textbf{347}
  (2009), no.~15-16, 849--852. \MR{2542882 (2010h:16039)}

\bibitem[MS82]{MS}
A.~S. Merkur{\cprime}ev and A.~A. Suslin, \emph{{$K$}-cohomology of
  {S}everi-{B}rauer varieties and the norm residue homomorphism}, Izv. Akad.
  Nauk SSSR Ser. Mat. \textbf{46} (1982), no.~5, 1011--1046, 1135--1136.
  \MR{675529 (84i:12007)}

\bibitem[Rev77]{Revoy}
Ph. Revoy, \emph{Alg\`ebres de {C}lifford et alg\`ebres ext\'erieures}, J.
  Algebra \textbf{46} (1977), no.~1, 268--277. \MR{0472881 (57 \#12568)}

\bibitem[Ros99]{Rost}
Markus Rost, \emph{The chain lemma for {K}ummer elements of degree 3}, C. R.
  Acad. Sci. Paris S\'er. I Math. \textbf{328} (1999), no.~3, 185--190.
  \MR{1674602 (2000c:12003)}

\bibitem[Siv12]{Siv}
A.~S. Sivatski, \emph{The chain lemma for biquaternion algebras}, J. Algebra
  \textbf{350} (2012), 170--173. \MR{2859881 (2012j:16037)}

\end{thebibliography}
\end{document}